\documentclass[11pt, reqno]{amsart}
\usepackage{syntonly}
\usepackage{amsmath,amsfonts,amssymb,amsthm,euscript,upgreek}
\usepackage{enumerate}
\usepackage{multirow}
\usepackage{appendix}

\newcommand{\R}{\mathbb{R}}

\newcommand{\CP}{\mathbb{C}\mathrm{P}}

\newcommand{\C}{\mathbb{C}}
\newcommand{\Z}{\mathbb{Z}}

\newcommand{\Ric}{\mathrm{Ric}}

\newcommand{\K}{K\"{a}hler}
\newcommand{\Sz}{Szeg\"{o}}

\newtheorem{theor}{Theorem}

\newtheorem{lem}[theor]{Lemma}
\newtheorem{cor}[theor]{Corollary}

\newtheorem{ex}{Example}
\newtheorem{remar}[theor]{Remark}

\begin{document}
\title[SZEG\"{O} kernel,  regular quantizations...] {SZEG\"{O}  kernel, regular quantizations and spherical CR-structures}

\author{Claudio Arezzo}
\address{Abdus Salam International Center for Theoretical Physics \\
                  Strada Costiera 11 \\
         Trieste (Italy) and Dipartimento di Matematica \\
         Universit\`a di Parma \\
         Parco Area delle Scienze~53/A  \\
         Parma (Italy)}
\email{arezzo@ictp.it}

\author{Andrea Loi}
\address{Dipartimento di Matematica \\
         Universit\`a di Cagliari}
         \email{loi@unica.it}

\author{Fabio Zuddas}
\address{Dipartimento di Matematica e Informatica \\
          Via delle Scienze 206 \\
         Udine (Italy)}
\email{fabio.zuddas@uniud.it}

\thanks{
The first author was supported  by the M.I.U.R. Project \lq\lq Geometric
Properties of Real and Complex Manifolds'' and by ESF within the program \lq\lq Contact and Symplectic Topology"}
\date{}

\begin{abstract}
We compute the \Sz\  kernel of the unit circle bundle of a  negative line bundle dual to a regular quantum  line bundle over a compact \K\ manifold.
As a corollary we provide an infinite family of smoothly bounded strictly pseudoconvex domains  on complex manifolds
(disk bundles over homogeneous Hodge manifolds) for which the log-terms in the Fefferman expansion of the \Sz\ kernel vanish and which are not locally CR-equivalent to the sphere. We also give a  proof of the fact that, for homogeneous Hodge manifolds, the existence of a locally spherical CR-structure on the unit circle bundle alone implies that the manifold is biholomorphic to a projective space. Our results  generalize those obtained by  M. Engli\v{s} and G. Zhang in  \cite{englisramadanov} for Hermitian symmetric spaces of compact type. 
\end{abstract}

\maketitle

\section{Introduction}
Let $(L, h)$ be a positive Hermitian line bundle over a compact \K\ manifold $(M, g)$ of complex dimension $n$, such that $\Ric (h)=\omega_g$,, where $\omega_g$ denotes the \K\ form associated to $g$
and $\Ric (h)$  is the two--form on $M$ whose
local expression is given by
$\Ric (h)=-\frac{i}{2}
\partial\bar\partial\log h(\sigma (x), \sigma (x)),$
for a trivializing holomorphic section $\sigma :U\rightarrow
L\setminus\{0\}$. In the quantum mechanics terminology the pair $(L, h)$ is called a {\em geometric quantization} of $(M, \omega_g)$ and $L$ the {\em quantum line bundle}.
Consider the negative Hermitian line bundle $(L^*,h^* )$ over $(M, g)$ dual   to $(L, h)$ and
let $D\subset L^*$ be the unit disk bundle over $M$, i.e.
\begin{equation}
D=\{v\in L^* \ |\  \rho (v):=1-h^*(v, v)>0\}.
\end{equation}
It is not hard to see (and well-known)  that the condition $\Ric (h)=\omega_g$ implies that $D$ is a strongly pseudoconvex   domain in $L^*$ with smooth boundary
$X=\partial D=\{v\in L^*\ |\ \rho (v)=0\}$. $X$ will be called  the  {\em unit circle bundle}.

Consider the  separable Hilbert space
${\mathcal H}^2(X)$ consisting
of  all holomorphic functions $f: L^*\rightarrow\C$ with scalar product given by
$\int_Xf\bar gd\mu$
where $d\mu=\alpha\wedge (d\alpha)^n$ and 
$$\alpha =-i\partial\rho_{|X}=i\bar\partial\rho_{|X}$$
 is the standard contact form on $X$ associated to the strongly pseudoconvex domain $D$
\footnote{Notice that the holomorphic functions of $L^*$ are automatically $L^2$-integrable on $X$ due to the compactness of $M$ and  of $X$.}.

Let $\{f_j\}_{j=1, \dots}$ be an orthonormal basis of ${\mathcal H}^2(X)$, i.e.  
$$\int_Xf_j\bar {f_k}d\mu=\delta_{jk}.$$

The  Szeg\"{o} kernel of $D$ is defined by:
$${\mathcal S}(v)=\sum_{j=1}^{+\infty} f_j(v)\overline {f_j(v)},\  v\in D.$$

A fundamental result of C. Fefferman  \footnote{Originally proved  for the boundary singularities of the  Bergman kernel of a  strongly pseduconvex domain in $\C^n$.}  \cite{F} asserts  that 
there exist $a, b\in C^{\infty} (\bar D)$, $a\neq 0$ on $X=\partial D$ such that:
\begin{equation}\label{fefferman}
{\mathcal S}(v)=a(v)\rho(v)^{-n-1}+b(v)\log\rho (v), \ v\in D
\end{equation}
where $\rho (v)=1-h^*(v, v)$ is the defining function of $D$.

The function $b(v)$ in (\ref{fefferman}) is called the {\em logarithmic term}  ({\em log-term} from now on) of the Szeg\"{o} kernel.  One says that the log-term  of the Szeg\"{o} kernel of the disk  bundle $D\subset L^*$ vanishes if $b=0$.

\begin{ex}\label{fundexample}\rm
The basic example here  is  the complex projective space $\CP^n$ equipped with the 
Fubini--Study \K\ form $\omega_{FS}$,  $L=O(1)$  the hyperplane bundle and $L^*=O(-1)$ the canonical bundle over
$\CP^n$. The 
corresponding unit  circle bundle $X\rightarrow \C P^n$ is CR-equivalent to $S^{2n+1}$
via the map $S^{2n+1}\rightarrow X, z\mapsto (\C z, z)$ and 
the restriction to $X$ of the  projection $\pi :L^*\rightarrow \CP^n$  is the Hopf fibration. Morever, a direct computation shows that the log-term of the Szeg\"{o} kernel of the disk bundle vanishes. More generally,  for any nonnegative integer $m$, one considers the line bundle $L^m=O(m)$ over $\CP^n$ and its dual $L^{*m}=O(-m)$. In this case the unit circle bundle $X$ is  CR-equivalent to the   lens space $S^{2n+1}/  \Z _m$. The diffeomorphism this time is given by  the map $z\mapsto (\C z, \otimes^m z)$ from the sphere $S^{2n+1}$  which induces a diffeomorhism form 
 $S^{2n+1}/  \Z _m$ onto $X$ (see e.g.
 \cite[p. 542]{komuro}).
 Also in this case the log-term of the Szeg\"{o} kernel of the disk bundle vanishes
 (this follows by a direct computation or as a very special case of  Theorem \ref{corolmainteorreg} below).
 \end{ex}
 
Therefore the  following questions  naturally arise.

 \vskip 0.3cm
 
\noindent
{\bf Question 1.}
Let $D$ be as above. Suppose that the log-term of the \Sz\ kernel of $D$ vanishes. Is it the boundary $X$ of 
$D$ locally CR-equivalent to the sphere?

 \vskip 0.3cm
 
\noindent
{\bf Question 2.}
Assume that   $X$ is (locally) CR-equivalent to the sphere. Is it true that $M$ is biholomorphic to $\CP^n$?

\vskip 0.3cm

Question 1 is a reasonable question since the singularity of the  \Sz\ kernel  is determined locally by the CR-structure of the boundary. This question is analogous to  Ramadanov's conjecture  asserting that if $\Omega\subset\C^n$  is a  strongly pseudoconvex domain with smooth boundary such that  its Bergman kernel has no logarithmic term then $\Omega$ is biholomorphic  to the unit ball in $\C^n$ \footnote{Observe that in  our situation the disk bundle $D$
 is not  even   homotopically equivalent to a disk in $\C ^n$ since the  the homotopy type of  the disk bundle is the same of its base (compact) manifold  $M$.}.
Questions $1$ and $2$ are inspired by (and implicitly contained in) the paper  of   Z. Lu and G. Tian \cite{logterm} (see also the next section)  and  by the work of M. Engli\v{s} and G. Zhang  \cite{englisramadanov} who 
provides a negative answer to Question 1 and a positive answer to Question 2  when  the disk bundle $D\subset L^*$ arises from  a positive line bundle $L$ over an Hermitian symmetric space of compact type. 
 Indeed, in this case  he computes  the Szego kernel of $D$, proves that  the log-term vanishes and that $X=\partial D$ is locally spherical iff $M=\CP^n$.
In the same paper he asks (Question 4 in \cite{englisramadanov}) if there were other cases where the log-term of the Szeg\"{o} kernel of the disk bundle vanishes.

One of the aim of the paper is to show that there are indeed many other cases where this occurs. This is   expressed by the following theorem and its corollary which is the first result of the paper (the reader is referred to the next section for details and references on regular quantizations). 

\vskip 0.5cm

 \begin{theor}\label{mainteorreg}
Let $(L, h)$ be a regular quantization of a compact \K\ manifold $(M, \omega)$ and let $X\rightarrow M$ be
the unit circle bundle of $(L^*, h^*)$.
Then the log-term of the  \Sz\ kernel  of $D$ vanishes.
\end{theor}

Notice that a  compact homogeneous Hodge manifold (i.e. a simply-connected compact and homogeneous  \K\ manifold)
admits a regular quantization  (see  next section).
Hence, we get the following corollary of Theorem \ref{mainteorreg}.
\begin{cor}\label{corolmainteorreg}
Let $(M, g)$ be a  homogeneous Hodge manifold  and  let $(L, h)$ be the Hermitian line bundle over $M$ such that $\Ric (h)=\omega_g$. Then the log-term of the Szeg\"{o} kernel of the disk bundle $D\subset L^*$ vanishes.  
 \end{cor}
 It is worth pointing out that it  is still an open problem to understand if there exist non-homogeneous \K\ manifolds which admit a regular quantization. Nevertheless the set of compact homogeneous Hodge manifolds  strictly includes the set of Hermitian symmetric spaces of compact type and so our  corollary provides new  examples, different from the Hermitian symmetric spaces, where the log-term vanishes.

Turning our attention to Question $2$, we prove  that, among all homogeneous spaces, $\CP^n$ is the only one for which 
  $X=\partial  D$ is locally CR-equivalent  to the sphere $S^{2n+1}$ at some point. This is  expressed by  the following theorem which represents  our second result and generalizes the above mentioned results of  M. Engli\v{s} and G. Zhang  valid for Hermitian symmetric  spaces of compact type. 
 \begin{theor}\label{homsph}
 If $M$ is a homogeneous Hodge manifold and $X$ has a CR-structure locally equivalent to $S^{2n+1}$ at same point, then $M$ is biholomorphic to the complex projective space and $L$ is a multiple of the tautological line bundle.
  \end{theor}
 
We finally point out that the papers of  X. Huang (\cite{Hua}, \cite{Hua2}, \cite{HJ}) dealing with  the links of complex analytic spaces with an isolated singularity carrying a spherical CR-structure could be useful to attack Question 2 without the assumption of homogeneity.  Indeed,  by applying  Huang's results to the unit circle bundle of an embedded projective manifold $M$, seen as a link of the affine cone over $M$, one obtains some informations  on the CR-structure of the cone and consequently on the manifold $M$. Unfortunately,  these informations seem not to be enough to provide 
a positive  answer to Question 2.

\vskip 0.3cm

The paper  contains other two sections. In the first one  we recall the basic material on Kempf distortion function, TYZ expansion  and  regular quantizations while the  last section is dedicated to the proof of Theorem \ref{mainteorreg} and Theorem \ref{homsph}.

  

\section{Kempf distortion function,  Szeg\"{o} kernel and regular quantizations}\label{secreg}
Let $(L, h)$ be a positive Hermitian line bundle over a compact \K\ manifold $(M, g)$ of complex dimension $n$, such that $\Ric (h)=\omega_g$ as in the introduction.
Let
$m\geq 1$ be an integer and
consider the Kempf distortion function associated to $mg$, i.e.
\begin{equation}\label{Tm}
T_{mg} (x) =\sum_{j=0}^{d_m}h_m(s_j(x), s_j(x)).
\end{equation}
where  $h_m$ is an hermitian metric   on
$L^m$ such that $\Ric(h_m)=m\omega_g$
and   $s_0, \dots , s_{d_m}$, $d_m+1=\dim H^0(L^m)$ is  an orthonormal basis of $H^0(L^m)$
(the space of holomorphic sections of $L^m$)
with respect to the $L^2$-scalar product
$$\langle s, t \rangle_m =\int_Mh_m(s(x), t(x))\frac{\omega_g^n(x)}{n!}, \ s, t\in H^0(L^m).$$
(In the  quantum geometric context
$m^{-1}$ plays the role of Planck's constant, see e.g.
\cite{arlquant}).

As suggested by the notation this function depends only on the
\K\ metric $g$ and on $m$ and not on the orthonormal basis chosen.

The function $T_{mg}$ has appeared in the literature under different
names. The earliest one was probably the $\eta$-function of J.
Rawnsley \cite{ra1} (later renamed to $\theta$ function in
\cite{cgr1}), defined for arbitrary (not necessarily compact) K\"{a}hler manifolds,
followed by the {\em distortion function } of G. R. Kempf \cite{ke} and S. Ji
\cite{ji}, for the special case of Abelian varieties and of S. Zhang
\cite{zha} for complex projective varieties.
The metrics for which
$T_{g}$ is constant (i.e. $T_{mg}$ is constant for $m=1$) were called {\em critical} in \cite{zha} and {\em
balanced}  in \cite{do} (see  also \cite{hombal})  where S. Donaldson studies  the link between the  existence and uniqueness of  balanced metrics in a fixed  cohomology class and constant scalar curvature metrics.

One can give a
quantum-geometric interpretation of $T_{mg}$   as follows.
Take $m$ sufficiently large such  that for each point $x\in M$
there exists $s\in H^0(L^m)$ non-vanishing at $x$ (such an $m$ exists by standard algebraic geometry methods
and corresponds to the free-based point condition in Kodaira's theory).  Consider the
so called  {\em coherent states map}, namely the
holomorphic map of $M$ into the complex projective space
${\C}P^{d_m}$ given by:
\begin{equation}\label{psiglob}
\varphi_m :M\rightarrow {\C}P^{d_m}:
x\mapsto [s_0(x): \dots :s_{d_m}(x)].
\end{equation}
One can prove (see, e.g.  \cite{arlcomm}) that
\begin{equation}\label{obstr}
\varphi ^*_m\omega_{FS}=m\omega_g +
\frac{i}{2}\partial\bar\partial\log T_{mg} ,
\end{equation}
where $\omega_{FS}$ is the Fubini--Study form on
${\C}P^{d_m}$, namely the  \K\ form which in homogeneous
coordinates $[Z_0,\dots, Z_{d_m}]$ reads as \linebreak
$\omega_{FS}=\frac{i}{2}\partial\bar\partial\log \sum_{j=0}^{d_m}
|Z_j|^2$.
Since the equation  $\partial\bar\partial f=0$
implies that $f$ is constant,  it follows  by (\ref{obstr})
that   a  metric  $mg$  is balanced if and only if
it is projectively induced (via the coherent states map).
Recall that   a   K\"{a}hler metric $g$ on a complex manifold $M$
 is
 {\em projectively induced} if there exists a positive  integer $N$ and a  holomorphic map  $\psi : M\rightarrow
{\C}P^N$, such that  $\psi^*g_{FS}=g$
(the author is referred to the seminal paper of E. Calabi \cite{ca} for more details on
the subject).
Not all K\"{a}hler metrics are balanced or projectively induced.
Nevertheless, G. Tian \cite{ti0} and  W. D. Ruan \cite{ru} solved  a conjecture posed by Yau
by showing that
$\frac{\varphi_m ^*g_{FS}}{m}$ $C^{\infty}$-converges to $g$.
 In other words, any polarized
metric on a compact complex manifold is the $C^{\infty}$-limit of
(normalized) projectively induced K\"{a}hler metrics. S. Zelditch
\cite{ze} generalized the Tian--Ruan theorem by proving a complete
asymptotic expansion (called TYZ expansion) in the $C^\infty$ category, namely
\begin{equation}\label{asymptoticZ}
T_{mg}(x) \sim \sum_{j=0}^\infty  a_j(x)m^{n-j},
\end{equation}
where  $a_j(x)$, $j=0,1, \ldots$, are smooth coefficients with $a_0(x)=1$.
More precisely.
for any nonnegative integers $r,k$ the following estimate holds:
\begin{equation}\label{rest}
||T_{mg}(x)-
\sum_{j=0}^{k}a_j(x)m^{n-j}||_{C^r}\leq C_{k, r}m^{n-k-1},
\end{equation}
where $C_{k, r}$ are constants depending on $k, r$ and on the
K\"{a}hler form $\omega_g$ and $ || \cdot ||_{C^r}$ denotes  the $C^r$
norm in local coordinates.
 Later on,  Z. Lu \cite{lu} (see also \cite{liulu}),  by means of  Tian's peak section method,
 proved  that each of the coefficients $a_j(x)$ in
(\ref{asymptoticZ}) is a polynomial
of the curvature and its
covariant derivatives at $x$ of the metric $g$ which can be found
 by finitely many algebraic operations.
 Furthermore,  he explicitely computes
$a_j$ with $j\leq 3$.

Notice that   prescribing the values of  the  coefficients
of the  TYZ expansion gives rise to interesting elliptic PDEs
as shown by Z. Lu and G. Tian  \cite{logterm}. The main result obtained there  is that if the log-term
of the Szeg\"{o} kernel  of the unit disk bundle over $M$ defined in the introduction  vanishes then  $a_k=0$, for  $k>n$
($n$ being the complex dimension of $M$). Moreover Z. Lu has conjectured (private communication) that the converse is true, namely if 
$a_k=0$, for  $k>n$ then the log-term vanishes.

\vskip 0.5cm

In this paper we are interested in those \K\ 
manifolds  admitting  a regular quantization.
A \K\ manifold  $(M, g)$ admits a regular quantization if  
there exists  a positive Hermitian line bundle $(L, h)$ as above  such that 
the Kempf distortion $T_{mg} (x)$ is  a constant $T_{mg}$ (depending on $m$)
for all non-negative integer   $m\geq 1$.  In other words $(M, g)$ admits  a regular quantization iff 
the family of \K\ metrics $mg$ is balanced for all $m$. Regular quantizations play a prominent role in the study of Berezin quantization of \K\ manifolds  (see  \cite{arlquant} and  \cite{hombal} and reference therein).
From our point of view we are interested to the following two facts:
\begin{itemize}
\item
 if a   \K\ manifold  $(M, g)$ is homogeneous, i.e. the group of holomorphic isometries acts transitively on it, then $(M. g)$ admits  a regular quantization (see, e.g.   \cite[Theorem 5.1]{arlquant} for a proof).
\item
if a  \K\ manifold admits a regular quantization then it follows by the very definition of Kempf distortion function  that  
\begin{equation}\label{Tmh}
T_{mg}(x)=T_{mg}=\frac{h^{0}(L^m)}{V(M)}
\end{equation}
and so by Riemann--Roch theorem  $T_{mg}$ is a  monic  polynomial in  $m$  of degree $n$.
\end{itemize}
 Thus in the regular case  (and hence in the homogeneous case)
the TYZ expansion is finite (being a polynomimal), the coefficients 
$a_k$ are constants for all $k$ and   $a_k=0$ for $k>n$.
Hence in the light of Z. Lu conjecture mentioned above one could believe that the Szeg\"{o} kernel of the disk bundle has vanishing log-term in the regular situation.  Our  main Theorem \ref{mainteorreg} shows that this is indeed the case.

\section {The proofs of the main results}
\begin{proof}[Proof of Theorem \ref{mainteorreg}]
First notice that the Hardy space ${\mathcal H^2}(X)$ defined  in the introduction admits the Fourier decomposition into irreducible factors with respect to  the natural  $S^1$-action. More precisely, 
$${\mathcal H}^2(X)=\oplus_{m=0}^{+\infty}{\mathcal H^2_m}(X)$$
where 
${\mathcal H^2_m}(X)=\{f\in {\mathcal H}^2(X) \ | \ f(\lambda v)=\lambda^mf(v),\  \lambda\in S^1 \}$.
S. Zelditch \cite{ze} shows that the map $s\in H^0(L^m)\mapsto \hat s\in {\mathcal H}_m^2(X)$ given by:
$$\hat s(v)=v^{\otimes m}(s(x)),\ v\in X, \ x=\pi(v),\ \pi:L^*\rightarrow M$$
is  an isometry
between $H^0(L^m)$ and ${\mathcal H}_m^2(X)$. Moreover, it  not hard to see that, for every 
 $v\in X$,
\footnote{It is worth pointing out that is exactly formula (\ref{smtm}) together with
 Boutet de Monvel-Sj\"{o}strand parametrix for the Szeg\"{o} kernel 
which allows S. Zelditch \cite{ze} to get the TYZ expansion of Kempf's distortion function.}
\begin{equation}\label{smtm}
{\mathcal S}_m(v):=\sum_{j=0}^{d_m}\hat s_j(v)\overline{\hat s_j (v)}=\sum_{j=0}^{d_m}h_m (s_j(x), s_j(x))=T_{mg}(x), \ x=\pi(v)\in M,
\end{equation}
where $T_{mg}(x)$ is the  Kempf distortion function.
Thus
$${\mathcal S} (v )=\sum_{m=0}^{+\infty}{\mathcal S}_m(v)=\sum_{m=0}^{+\infty}(h_*(v, v))^mT_{mg}(x),\ v\in D, x=\pi (v).$$
Since, by assumption, the quantization $(L, h)$ of $(M, \omega_g)$ is regular one gets by (\ref{Tmh}):
$${\mathcal S}(v)=\sum_{m=0}^{+\infty}\frac{h^{0}(L^m)}{V(M)}(h_*(v, v))^m,$$
where $V(M)=\int_M\frac{\omega_g^n}{n!}$ denotes the volume of $M$.

As we have already observed in the previous section, by Riemann--Roch $h^0(L^m)$ is a monic  polynomial in  $m$  of degree $n$ so it can be written as  linear combination
of the binomial coefficients  $C_{k}^{m+k}=\frac{(m+k)!}{m!k!}$, namely
$$h^0(L^m)=\sum_{k=0}^nd_kC_{k}^{m+k},\  d_n=n!.$$
Hence
$${\mathcal S} (v )=\frac{1}{V(M)}\sum_{k=0}^nd_k\sum_{m=0}^{\infty}C_{k}^{m+k}(h_*(v,  v))^m.$$
By using the fact that 
$$\sum_{m=0}^{\infty}C_{k}^{m+k}x^m=\frac{1}{(1-x)^{k+1}}, \ 0<x<1$$
and $\rho (v)=1-h^*(v, v)$ one gets:
$${\mathcal S} (v )=
\frac{1}{V(M)}\sum_{k=0}^nd_k [\rho (v)]^{-k-1},$$
Therefore
$${\mathcal S}(v)=a(v)\rho(v)^{-n-1}$$
where 
$$a(v)=V(M)^{-1} n!+V(M)^{-1}\sum_{k=0}^{n-1} d_k[\rho(v)]^{n-k}.$$
It then follows  (cfr. (\ref{fefferman}) above) that  the log-term of the  Szeg\"{o} kernel  of the unit disk bundle $D\subset L^*$ vanishes and we are done.
\end{proof}

In order to prove Theorem \ref{homsph} we 
need the following cohomological lemma.

\begin{lem}\label{cohomhom}
Let $(M, \omega )$ be a homogeneous Hodge manifold. Assume that the Betti numbers of $M$ satisfy $b_{2j-2}=b_{2j}$ for $j=1, \dots ,n$, where we are taking the  dimension of the corresponding  cohomology groups with real coeffcients. Then $M$ is biholomorphic either to  $\C P^n$ or to the complex odd  quadric
$$Q_n=\{[Z_0, \dots , Z_{n+1}]\in \C P^{n+1}\ | \ Z_0^2+\cdots +Z_{n+1}^2=0\}, \ n=2p-1, \ n>1 .$$
\end{lem}
\begin{proof}
Let us recall that a compact  homogeneous Hodge manifold $M$ can be always realized as a flag manifold $G/K$, where $G$ is a semisimple Lie group and $K$ is the stabilizer of an element $x \in Lie(G)$ for the adjoint action. Such a manifold can be described combinatorially by a {\it painted Dynkin diagram} $\Gamma$, that is a Dynkin diagram of a complex semisimple Lie group (see for example \cite{knapp}) with some nodes painted in black so that by deleting the black nodes from $\Gamma$ one gets the Dynkin diagram of the semisimple part of $K$. One also endows the diagram with an {\it equipment}, that is a bjection between the nodes of $\Gamma$ and a basis $\Delta = \{\alpha_1, \dots, \alpha_m \}$ of the root system $R$ of $Lie(G)$, so that the elements of $\Delta$ associated to the white nodes are a basis of the root system $R_K$ of the semisimple part of $Lie(K)$.  In the following, we will call {\it black roots} the elements of $R \setminus R_K$.  By definition of basis, every $\alpha \in R$ can be written as a combination $\alpha = \sum_{i=1}^m k_i \alpha_i$, where the $k_i$'s belong to $\Z$ and are either all non-negative or all non-positive. A non-zero root such that $k_i \geq 0$ for every $i=1, \dots, m$ is called a {\it positive root}. The set of positive roots will be denoted $R^+$. For every $\alpha = \sum_{i=1}^m k_i \alpha_i \in R^+$, the {\it height} of $\alpha$ is the integer $h(\alpha )= \sum_{i=1}^m k_i$. Let $R_K^+ = R_K \cap R^+$. The following formula for the Poincar\'e polynomial of $G/K$ can be found in \cite{aky}:

\begin{equation} \label{poincareserdica}
P(G/K, t^{1/2}) = \prod_{\alpha \in R^+ \setminus R_K^+} \frac{1 - t^{h(\alpha)+1}}{1 - t^{h(\alpha)}}.
\end{equation}

In order to prove the lemma we will use (\ref{poincareserdica}) to determine all the painted Dynkin diagrams for which 

\begin{equation} \label{equaztesi}
P(G/K, t^{1/2}) = c + c t+ c t^2 + \dots + c t^n
\end{equation}

First, notice that by evaluating both members of (\ref{equaztesi}) at $t=0$ we get $c=1$. We shall now prove by induction on  $m$  that, under the assumptions, for every $1 \leq m \leq n$ there exists exactly one root in $R^+ \setminus R_K^+$ of height $m$.

\noindent In order to do that, let us denote $f_{\alpha} =  \frac{1 - t^{h(\alpha)+1}}{1 - t^{h(\alpha)}}$ for every $\alpha \in R^+ \setminus R_K^+ $. If $k = \sharp (R^+ \setminus R_K^+)$, we then have $P(t) = P(G/K, t^{1/2}) = \prod_{i=1}^k f_{\alpha_i}(t)$. 

\noindent By a straight calculation, one easily checks that the i-th derivative of $f_{\alpha}$ at $t=0$ vanishes for $0 < i < h(\alpha)$ and equals $i!$ for $i = h(\alpha)$. By (\ref{equaztesi}) we have $P'(0) = \sum_{i=1}^k f_{\alpha_i}'(0) =1$, which clearly implies the claim for $m=1$. Now, let us assume that the claim is true for $m = 1, \dots, s-1$. Then, up to reordering the roots, we have

\begin{equation} \label{induction1}
P(t)= (1 + t + \cdots t^{s-1}) \prod_{i=s}^k f_{\alpha_i}(t)
\end{equation}

\noindent where $h(\alpha_i) \geq s$ for every $i =s, \dots, k$. This implies that

\begin{equation} \label{induction2}
P^{(s)}(0)= \frac{d^s}{dt^s}|_{t=0} \prod_{i=s}^k f_{\alpha_i}(t) = \sum_{i_s + \cdots + i_k = s} \frac{s!}{i_s! \cdots i_k!} f_{\alpha_s}^{(i_s)}(0) \cdots f_{\alpha_k}^{(i_k)}(0)
\end{equation}

\noindent and then, by (\ref{equaztesi}),

\begin{equation} \label{induction3}
\sum_{i_s + \cdots i_k = s} \frac{1}{i_s! \cdots i_k!} f_{\alpha_s}^{(i_s)}(0) \cdots f_{\alpha_k}^{(i_k)}(0) =1.
\end{equation}

 But, since $h(\alpha_j) \geq s$ and $i_j \leq s$ for every $j = s, \dots, k$,  we get $\frac{1}{s!} \sum_{j=s}^k f_{\alpha_j}^{(s)}(0) = 1$, which implies the claim for $m = s$ and then concludes the proof by induction.

In particular, for $m=1$, by definition of height it follows that the Dynkin diagram of $G/K$ must have only one black vertex. Thus, we are left with determining the position of this vertex in the diagram. More precisely, one can  conclude the proof  by showing that, for each of the complex semisimple Lie algebras, the only painted diagrams with only one black vertex and which do not admit two distinct black positive roots having the same height correspond either to $\C P^n$ or to the complex odd  quadric.
This can  be obtained  by a careful  checking of  the list of  the set of roots $R$ and a basis $\Delta$ for each of the complex semisimple Lie algebras (see, for example, Appendix C of \cite{knapp}).
The details are left to the reader.
\end{proof}

\begin{proof}[Proof of Theorem \ref{homsph}]

If $X$ were locally CR-equivalent to $S^{2n+1}$ at same point, then, by homogeneity,  this would be true  at every point, i.e. $X$ would be spherical. Since $X$ is compact and homogeneous,
Proposition 5.1 in Burns and Shnider \cite{burnsshnider} would then imply that 
$X$ is diffeomorphic to a lens space $S^{2n+1}/  \Z _m$ for some $m$. Therefore, by 
a Gysin sequence argument (see e.g. the proof of 
 Corollary 3.7 in \cite{englisramadanov}) one gets that  the Betti numbers of $M$  must satisfy the relations 
$b_{2j-2}=b_{2j}$ for $j=1, \dots ,n$. 
By Lemma \ref{cohomhom} $M$ is then  biholomorphic either  to 
$\CP^n$ or to 
$Q_{n}\subset\C P^{n+1}$, $n=2p-1>1$. The prove will be finished if we show that  the case of the odd-dimensional quadric  cannot occur when $n>1$. This follows by the fact that the unit  circle bundle over $Q_n$ is the Stiefel manifold $V_2(\R^{n+2})$
(see \cite[p. 1581]{boga})
 which is not diffeomorphic to $S^{2n+1}/  \Z _m$ for $n>2$
 (for example $H^2 (S^{2n+1}/  \Z _m, \Z)=\Z_m$ while
 $H^2 (V_2(\R^{n+2}), \Z)=0$, for $n>2$, see \cite{hatcher} and   \cite{borel}, respectively,  for the computation of the $\Z$-cohomolgy ring of the lens spaces and of the Stiefel manifolds). This ends the proof of the theorem.
 \end{proof} 

\begin{remar}\rm
In the proof of Theorem  \ref{homsph}  the relation between the  Betti numbers of $M$ (and the homogeneity) allows to  deduce that our manifold  $M$  is  biholomorphic to a complex projective space.
We want to point out that there exist complex  algebraic  surfaces
with the same Betti numbers of  a complex projective but not biholomorphic to it. 
These are the celebrated {\em fake projective spaces} (see, e.g.  \cite{fake}, for details). Such a  \K\ surface $(F, g_F)$ is 
 obtained by taking the quotient of the unit ball $B^2\subset \C^2$ by a subgroup of biholomorphisms $\Gamma$ of $B^2$, i.e. $F=B^2/\Gamma$ and the metric $g_F$ is the unique (up to homotheties)  \K -Einstein metric on $F$ with negative scalar curvature whose pull-back to  $B^2$ is given by the hyperbolic metric  $g_{hyp}$.

Now, consider a geometric   quantization $(L, h)$ for $(F, g_F)$ and the corresponding Kempf distortion function $T_{mg_F}$. It is not hard to see that the coefficients in the  TYZ expansion of  $T_{mg_F}$   are constants and, moreover, $a_k=0$, for $k>2$. Indeed this follows by the fact that,  by Lu's theorem the coefficients depend on the curvature and are the same as those of $(B^2, g_{hyp})$ for which the computation  is  well-known. Hence, if one believes the validity of   Lu's  conjecture mentioned above (see the end of Section \ref{secreg}), $(F, g_F)$ would be an example of  projective algebraic  surface  
with the same Betti numbers of the projective space, with vanishing log-term  not biholomorphic to  $\CP^2$. Observe that in  this case even if all  coefficients of TYZ  expansion are  constants the quantization is not regular  and so the vanishing of the log-term cannot be deduced by  Theorem \ref{mainteorreg}. Indeed if  the quantization were regular then $g_F$ would be projectively induced and the same would be true for  the metric $g_{hyp}$ on $B^2$, in contrast with a well-known result of Calabi \cite{ca}.
\end{remar}

\end{document}